\documentclass[a4paper,12pt]{amsart}
\usepackage{amsthm,amsmath}
\usepackage{amssymb}
\usepackage[all]{xypic}
\usepackage{enumerate}

\begin{document}

\swapnumbers
\newtheorem{theorem}{Theorem}[section]
\newtheorem{proposition}[theorem]{Proposition}
\newtheorem{corollary}[theorem]{Corollary}
\newtheorem{lemma}[theorem]{Lemma}
\newtheorem*{theorem*}{Theorem}
\newtheorem*{proposition*}{Proposition}
\newtheorem*{corollary*}{Corollary}
\newtheorem*{lemma*}{Lemma}\theoremstyle{definition}
\newtheorem{definition}[theorem]{Definition}
\newtheorem{punto}[theorem]{}
\newtheorem{example}[theorem]{Example}
\newtheorem{remark}[theorem]{Remark}
\newtheorem{remarks}[theorem]{Remarks}
\newtheorem*{definition*}{Definition}
\newtheorem*{remark*}{Remark}

\newcommand{\cat}[1]{\mathcal{#1}}
\newcommand{\unidad}[1]{\eta_{#1}}
\newcommand{\counidad}[1]{\delta_{#1}}
\newcommand{\stat}[1]{\mathrm{Stat}(#1)}
\newcommand{\adstat}[1]{\mathrm{Adst}(#1)}
\newcommand{\coker}[1]{\mathrm{coker}(#1)}
\newcommand{\rcomod}[1]{\mathcal{M}^{#1}}
\newcommand{\lcomod}[1]{{}^{#1}\mathcal{M}}
\newcommand{\rmod}[1]{\mathcal{M}_{#1}}
\newcommand{\lmod}[1]{{}_{#1}\mathcal{M}}
\newcommand{\coend}[2]{\mathrm{Coend}_{#1}(#2)}
\newcommand{\End}[3]{\mathrm{End}_{#1}^{#2}(#3)}
\renewcommand{\hom}[3]{\mathrm{Hom}_{#1}(#2,#3)}
\newcommand{\Hom}{\mathrm{Hom}}
\newcommand{\cotensor}[1]{\square_{#1}}
\newcommand{\functor}[1]{\mathbf{#1}}
\newcommand{\Mod}{\mbox{{\rm -Mod}}}
\newcommand{\cohom}[3]{\mathrm{h}_{#1}(#2,#3)}
\newcommand{\HOM}[3]{\mathrm{HOM}_{#1}(#2,#3)}
\def\|{\, | \,}
\def\C{\mathbb{C}\,} 

\newcommand{\END}{\mathrm{End}\, }

%

\title[Similarity, codepth two and QF bimodules]{Similarity, codepth two bicomodules and
QF bimodules}
\author{F.~Casta\~no Iglesias}
\address{Departamento de Estad\'{i}stica y Matem\'atica Aplicada \\ Universidad de
Almer\'{\i}a, Spain}
\email{fci@ual.es }
  \thanks{Research partially
supported by Spanish Project (MTM2005--03227) from MCT}
\author{Lars Kadison}
\address{Department of Mathematics \\
University of Pennsylvania \\
Philadelphia, PA 19104 U.S.A.}
 \email{lkadison@math.upenn.edu}
\date{}
\maketitle
\begin{abstract}
For any $k$-coalgebra $C$ it is shown that  similar quasi-finite
$C$-comodules  have strongly equivalent coendomorphism coalgebras;
 (the converse is in general not true). As an application we
give a general result about  codepth two coalgebra homomorphisms.
Also a notion of codepth two bicomodule is introduced.  The
last  section applies similarity to an endomorphism ring theorem for  quasi-Frobenius (QF) bimodules and then to finite depth ring extensions.
For QF extensions, we establish that left and right depth two are equivalent notions as well as a converse
endomorphism theorem, and characterize depth three in terms of separability
and depth two. 


\end{abstract}


\section*{Introduction}
For a ring $R$, two right $R$-modules $M$ and $N$ are \emph{similar}
\cite{A-F} \,($H$-equivalent in sense of Hirata) if $M$ is a direct summand of $N^{(n)}$ and $N$ is a
direct summand of $M^{(m)}$ for some $m,n\in \mathbb{N}$. For example, $M$ is
similar to $R_R$ if and only if $M$ is a \emph{progenerator}, i.e.,
$M$ is a  finitely generated projective generator. Hirata showed in \cite{hirata} that similar $R$-modules $M \sim N$ have Morita
equivalent endomorphism rings $E_M$ and $E_N$;
whence isomorphic centers, $\END {}_{E_M}M_R
\cong \END {}_{E_N}N_R$. (The converse is
not true in general.)
By taking $R$
to be the enveloping ring of two rings, we extend this to a notion of similar bimodules.

For a $k$-coalgebra $C$, the notion of \emph{ingenerator}
comodule is introduced by Lin in \cite{Lin} to characterize
\emph{strong equivalences} between comodule categories. A right
$C$-comodule $M$ is said to be an ingenerator if there are $m,n \in
\mathbb{N}$ and $P,Q\in \mathcal{M}^C$ such that $M\oplus P \cong
C^{(m)}$ and $C\oplus Q \cong M^{(n)}$. More generally, two right
$C$-comodules $M_C$ and $N_C$ are called \emph{similar} if there are
$m,n \in \mathbb{N}$ and $P,Q \in \mathcal{M}^C$ such that $M\oplus
P \cong N^{(m)}$ and $N\oplus Q \cong M^{(n)}$.
Our first result  in this note is to extend the result of  Hirata to coalgebras.
More precisely, if two quasi-finite $C$-comodules are similar, then their
coendomorphism coalgebras are strongly equivalent in sense of Lin
\cite{Lin}.
Using the notion of codepth two coalgebra homomorphism, introduced by the second author in \cite{kadison-coD2}, it is proved in Section 2, as an application,
  that any codepth two coalgebra homomorphisms $\varphi:C\rightarrow D$, such that $C_D$ and $(C\cotensor{D}C)_D$
   are quasi-finite comodules, have strongly equivalent coendomorphism coalgebras
    $e_{-D}(C)$ and $e_{-D}(C\cotensor{D}C)$.
    Also in Section 2 a notion of  \emph{codepth two  bicomodule} is introduced, where it is noted that
     $\varphi:C\rightarrow D$ is a codepth two coalgebra homomorphism if and only if the  $_DC_C$ is a codepth two bicomodule.

In the last  section, using similar bimodules, we establish for QF bimodules  an endomorphism
ring theorem.  In case of a QF ring extension, the tensor-square
is similar with its endomorphism rings.  This is applied
to depth two and depth three QF extensions, showing that several results in \cite{kadison-fd} generalize from Frobenius to QF extensions.


\section{Definitions and some results} Throughout this paper $k$ is a field and $\mathcal{M}_k$
stands for the category of $k$-vector spaces. A basic reference for
theory of coalgebras is, for example,  \cite{Sw}. A coalgebra over
$k$ is a $k$-space $C$ together with two $k$-linear maps $\Delta :C
\rightarrow C \otimes C$ (the unadorned tensor product is understood
to be over $k$) and $\epsilon : C \rightarrow k$ such that $(1_C
\otimes \Delta) \Delta = (\Delta \otimes 1_C) \Delta$ and $(1_C \otimes
\epsilon) \Delta = (\epsilon \otimes 1_C) \Delta = 1_C$.

A right $C$-comodule is a $k$-space $M$ with a $k$-map $\rho_{M}:M
\rightarrow M \otimes C$ such that $(\rho_{M} \otimes 1_C)\rho_{M} =
(1_M \otimes \Delta)\rho_{M}$ and $(1_C \otimes \epsilon)\rho_{M} = 1_M$.
 If $M$ and $N$ are $C$-comodules, a comodule map from $M$ to $N$ is a
$k$-map $f:M \rightarrow N$ such that $(f \otimes 1)\rho_{M} =
\rho_{N}f$. The $k$-space of all comodules maps from $M$ to $N$ is
denoted by $Com_{C}(M,N)$ and $\mathcal{M}^C$ denotes the category
of right $C$-comodules. In the same way we can construct the
category of left $C$-comodules $^C\mathcal{M}$.

It is well known that $\mathcal{M}^C$ is an abelian category. In
fact, $\mathcal{M}^C$ is a locally finite Grothendieck category
(generated by finite dimensional comodules). The fundamental
properties of the categories of comodules can be found in several
places, see e.g. \cite{T,Do}. Let $C$ be an arbitrary coalgebra, $M$
a right $C$-comodule and $N$ be a left $C$-comodule, the \emph{
cotensor product} $M\cotensor{C}N$ is the kernel of the $k$-map
$\rho_M\otimes 1-1\otimes \rho_N:M\otimes N\rightarrow M\otimes
C\otimes N.$ Following \cite{Do}, the cotensor product is a left
exact functor $\mathcal{M}^C\times ^C\mathcal{M}\rightarrow
\cat{M}_k$. Moreover, the mapping $m\otimes c\mapsto \epsilon (c)m$
and $c\otimes n \mapsto \epsilon (c)n$ yield a natural isomorphism
$M\Box_CC\cong M$ and $C\Box_CN\cong N.$ If $C$ and $D$  are two
coalgebras, $M$ is a $(C,D)$-bicomodule if $(1_C\otimes
\rho^+)\rho^- = (\rho^-\otimes 1_D)\rho ^+$ where $\rho^-:
M\rightarrow C\otimes M$ and $\rho^+ : M\rightarrow M\otimes D$ are
the structure maps of $M$. Moreover, if $X$ is a right $C$-comodule,
then the map $1_X\otimes \rho^+ : X\otimes M\rightarrow X\otimes
M\otimes D$ define over $X\otimes M$ a structure of right
$D$-comodule. In this case, $X\cotensor{C} M$ is a $D$-subcomodule
of $X\otimes M$. This defines a left exact functor
$-\cotensor{C}M:\mathcal{M}^C\rightarrow \mathcal{M}^D$ that
preserves direct sums (see \cite{T}).

Recall from \cite{T} that a right $C$-comodule $M$ is called
\emph{quasi-finite} if $Com_{C}(Y,M)$ is finite dimensional for
every finite dimensional comodule $Y\in \mathcal{M}^C$. This is
equivalent to the existence of a left adjoint $\cohom{-C}{M}{-}$,
called \emph{co-hom} functor, to $-\otimes M$. If $M$ is a
$(D,C)$-bicomodule, the functor $\cohom{-C}{M}{-}:
\mathcal{M}^C\rightarrow \mathcal{M}^D$ becomes a left adjoint to
the \emph{cotensor product} functor $-\cotensor{D}M
:\mathcal{M}^D\rightarrow \mathcal{M}^C$. If we assume that $M_C$ is
a quasi-finite comodule, then $e_{-C}(M) = \cohom{-C}{M}{M}$ is a
coalgebra, called co-endomorphism coalgebra of $M$. Furthermore, $M$
is a $(e_{-C}(M),C)$-bicomodule via $\theta_M : M\rightarrow
e_{-C}(M)\otimes M$, where $\theta: 1_{\mathcal{M}^C}\rightarrow
\cohom{-C}{M}{-}\cotensor{D}M$ denotes the unit of the adjunction.
\bigskip

Consider now  two quasi-finite right $C$-comodules $M_C$ and $N_C$.
If we denote by $D_M = e_{-C}(M)$ and $D_N = e_{-C}(N)$ their
coendomorphism coalgebras, then we can consider the diagram
\begin{equation}\label{composition}
\xymatrix@*+<14pt>{\mathcal{M}^{D_M}
 \ar[rr]^{-\cotensor{D_M}M}
   & & \mathcal{M}^C \ar@<1ex>[ll]^{\cohom{-C}{M}{-}}_{} \
 \ar[rr]^{\cohom{-C}{N}{-}}
   & &\mathcal{M}^{D_N} \ar@<1ex>[ll]^{-\cotensor{D_N}N}_{} }
\end{equation}
The composition of functors yield a pair of functors between the
comodules categories over the coendomorphism coalgebras:

$$\functor{F}= \cohom{-C}{N}{-\cotensor{D_M}M}:
\mathcal{M}^{D_M}\rightleftarrows \mathcal{M}^{D_N} :
\cohom{-C}{M}{-\cotensor{D_N}N} = \functor{G},$$ where
$\functor{F}(D_M)\cong\cohom{-C}{N}{M}$ and
$\functor{G}(D_N)\cong\cohom{-C}{M}{N}.$
\bigskip

Following \cite{T}, a \emph{Morita-Takeuchi context} $\Omega =
\left( C, D ; _CM_D , _DN_C ; f, g\right)$ consists of coalgebras
$C$ and $D$, bicomodules $_CM_D$ and $_DN_C$ and bicolinear maps
$f:C\rightarrow M\cotensor{D}N$ and $g:D\rightarrow N\cotensor{C}M$
satisfying the following commutative diagrams:
\[
\xymatrix{ M \ar^{\simeq}[rr] \ar_{\simeq}[d] & & M\cotensor{D}D
\ar^{I\cotensor{}g}[d] \\
C\cotensor{C}M \ar^{f\cotensor{}I}[rr] & &
M\cotensor{D}N\cotensor{C}M } \qquad \xymatrix{ N \ar^{\simeq}[rr]
\ar_{\simeq}[d] & & N\cotensor{C}C
\ar^{I\cotensor{}f}[d] \\
D\cotensor{D}N \ar^{g\cotensor{}I}[rr] & &
N\cotensor{C}M\cotensor{D}N }
\]
The context is said to be \emph{strict} if $f$ and $g$ are
bicolinear isomorphisms. In this case, the categories
$\mathcal{M}^{C}$ and $\mathcal{M}^{D}$ are \emph{equivalent} and we
say that $C$ is \emph{Morita-Takeuchi equivalent} to $D$. The strict condition of the
context is equivalent to say that the bicomodules $M$ and $N$ are
injective cogenerators  finitely cogenerated \cite[Theorem 2.5]{T}.
\bigskip

For any quasi-finite right $C$-comodules $M_C$ and $N_C$, it was
proved in \cite{Li-Wang} that
\begin{equation}\label{context}
\Gamma^C_{M,N} = \left(D_M, D_N ;\cohom{-C}{N}{M}, \cohom{-C}{M}{N};
f, g \right) \end{equation}
 is a Morita-Takeuchi
context with bicolinear maps $$f: D_M\rightarrow
\cohom{-C}{N}{M}\cotensor{D_N}\cohom{-C}{M}{N}$$ and $$g:
D_N\rightarrow \cohom{-C}{M}{N}\cotensor{D_M}\cohom{-C}{N}{M}$$
defined, respectively, by
$$(f\otimes 1)\theta_M = (1\otimes \theta_N)\overline{\theta}_M$$
$$(g\otimes 1)\overline{\theta}_N = (1\otimes\overline{\theta}_M)\theta_N$$
where $\theta$ and $\overline{\theta}$  are the units of the
adjoint pair $$(\cohom{-C}{M}{-}, -\cotensor{D_M}M)$$ and
$$(\cohom{-C}{N}{-}, -\cotensor{D_N}N),$$ respectively.

\section{Similar comodules and a strong equivalence}
In this section we consider the definition of similar comodules
and we prove that similar quasi-finite comodules  have strongly equivalent
coendomorphism coalgebras.
\begin{definition} Let $C$ be a $k$-coalgebra. Two right $C$-comodules $M$ and $N$ are
\emph{similar}, abbreviated $M\sim N$, if there are $m,n\in
\mathbb{N}$ and $C$-comodules $P$ and $Q$ such that $M\oplus P \cong
N^{(m)}$ and $N\oplus Q \cong M^{(n)}$.
\end{definition}
It is easy see that ``$\sim $'' defines an equivalence relation on
the class $\mathcal{M}^C$ of right $C$-comodules.  Notice  that if
$M_C$ is quasi-finite, the functor $\cohom{-C}{M}{-}$  preserves
 comodules similar to $M$, as does the functor $-\cotensor{D_M}M$, where
$D_M=e_{-C}(M)$ is the coendomorphism coalgebra of $M$. This
comes from the fact that the co-hom and cotensor functors preserve
direct sums. One obtains from \cite{Lin} that a right $C$-comodule
$M$ is an ingenerator of $\mathcal{M}^C$ if and only if $M$ is
similar  to $C_C$ as comodules. By transitivity of ``$\sim $'', we note that if there are several ingenerators in $\mathcal{M}^C$,
then all are similar. Therefore we can state a first result concerning similarity properties of comodules.
\begin{lemma}\label{ingenerator}
Assume that $M$ and $N$ are quasi-finite right $C$-comodules. If $M_C \sim N_C$, then
$\cohom{-C}{M}{N}$ and $\cohom{-C}{N}{M}$ are ingenerators in
$\mathcal{M}^{D_M}$ and $\mathcal{M}^{D_N}$, respectively.
\end{lemma}
\begin{proof} Suppose that $M_C \sim N_C$. Since the functor $\cohom{-C}{M}{-}$ (respectively,  $\cohom{-C}{N}{-}$)  preserve similar comodules to $M$ (respectively, to $N$), we deduce that  $D_M = \cohom{-C}{M}{M}\sim \cohom{-C}{M}{N}$ and $D_N = \cohom{-C}{N}{N}\sim \cohom{-C}{N}{M},$  hence   we obtain  the lemma.
\end{proof}
\bigskip

%

Let $C$ and $D$ be two $k$-coalgebras. Recall from \cite{Lin} that
$C$ is \emph{strongly equivalent} to $D$ if the category
$\mathcal{M}^C$ is equivalent to the category $\mathcal{M}^D$ via
inverse equivalences
$$\functor{F}: \mathcal{M}^C \rightleftarrows\mathcal{M}^D:
\functor{G},$$ such that $\functor{F}(C)$ is an ingenerator of
$\mathcal{M}^D$ and  $\functor{G}(D)$ is an ingenerator of
$\mathcal{M}^C$.
\bigskip

The  theorem below answers  our first aim in the affirmative.
\begin{theorem}\label{main} Let $C$ be a $k$-coalgebra and $M_C$ and $N_C$ be quasi-finite
$C$-comodules. If  $M_C$ is similar to $N_C$, then  $D_M$ is
strongly equivalent to $D_N$.
\end{theorem}
\begin{proof} Assume $M_C \sim N_C$. It follows from Lemma \ref{ingenerator} that $\cohom{-C}{M}{N}$ and $\cohom{-C}{N}{M}$ are ingenerators of
 $\mathcal{M}^{D_M}$ and $\mathcal{M}^{D_N}$, respectively.
By \cite[Theorem 2.5]{T}, the context $\Gamma^C_{M,N}$ in (2) is strict. So,  $D_M$ is Morita-Takeuchi equivalent to $D_N$. From diagram (\ref{composition}) the equivalence of categories is induced by the
composition functors $\functor{F}$ and $\functor{G}$. Moreover
$\functor{F}(D_M)\sim D_N$ and $\functor{G}(D_N)\sim D_M$, which
shows that the equivalence $\functor{F}:
\mathcal{M}^{D_M}\rightleftarrows \mathcal{M}^{D_N}: \functor{G}$ is
strong  in sense of \cite{Lin}. Hence the coendomorphism
coalgebras  are strongly equivalent. \end{proof}

\begin{remark} The converse fails in general since if $k$ is an
algebraically closed field  and  we consider two non isomorphic
simple comodules $S$ and $S'$ of $\mathcal{M}^{C}$, then
$\cohom{-C}{S}{S}= \mbox{Com}_C(S,S)^*= k^* \cong k$ and
$\cohom{-C}{S'}{S'}\cong k$. Thus $D_S$ and $D_{S'}$ are strongly
equivalent but $S$ and $S'$ cannot be similar.
\end{remark}
\begin{punto}
{\bf Application to codepth two coalgebra homomorphisms.} Let
$\varphi:C\rightarrow D$ be a homomorphism of coalgebras over a
field. Then $C$ has an induced $(D,D)$-bicomodule structure and any
$C$-comodule becomes a $D$-comodule via the corestriction functor
$(-)_{\varphi}:\rcomod{C}\rightarrow \rcomod{D}$ (see \cite{T}). A well-known result in the theory of coalgebras is that the corestriction functor
has the coinduction functor $-\cotensor{D}C:
\rcomod{D}\rightarrow \rcomod{C}$ as right adjoint. Thus we have the adjoint couple
of functors

\begin{equation}
\xymatrix@*+<14pt>{\mathcal{M}^{D}
 \ar[rr]^{-\cotensor{D}C}
   & & \mathcal{M}^C \ar@<1ex>[ll]^{(-)_{\varphi}}_{}
  }
\end{equation}
 A coalgebra homomorphism $\varphi:C\rightarrow D$ is called \emph{left codepth two} \cite[Definition 6.1]{kadison-coD2} if the cotensor product $C\cotensor{D}C$ is isomorphic to a direct summand of a finite direct sum of $C$ as  $(D,C)$-bicomodules. Since $C$ is in general isomorphic to  a direct summand of $C\cotensor{D}C$ as $(D,C)$-bicomodules,  we note that $\varphi$ is a left codepth two coalgebra homomorphism if $C\cotensor{D}C$ and $C$ are similar as $(D,C)$-bicomodules.
Right codepth two coalgebra homomorphisms are similarly defined.
Applying Theorem \ref{main} to codepth two coalgebra homomorphism, we obtain:
\begin{theorem*}  Let $\varphi:C\rightarrow D$ be a codepth two coalgebra homomorphism. If $C_D$ and $( C\cotensor{D}C)_D$ are quasi-finite right $D$-comodules, then the coendomorphism coalgebras $e_{-D}(C)$ and $e_{-D}(C\cotensor{D}C)$ are strongly equivalent.
\end{theorem*}
Consider now a $(D,C)$-bicomodule $M$ with $M_C$ quasi-finite. We have the adjoint pair $$(\cohom{-C}{M}{-}, -\cotensor{D}M).$$
A notion of codepth two bicomodule is the following.
\begin{definition*}
A $(D,C)$-bicomodule $M$ where $M_C$ is quasi-finite is said to be \emph{codepth two bicomodule} if $\cohom{-C}{M}{ e_{-C}(M)\cotensor{D}M}\sim e_{-C}(M)$ as $(C,D)$-bicomodules.
\end{definition*}
Note that $\varphi:C\rightarrow D$ is a codepth two coalgebra homomorphism if and only if $_DC_C$ is a codepth two bicomodule.
\end{punto}
\section{Applications of similarity to QF bimodules and Finite Depth Extensions}
In this section, we recall some known notions and provide several new results concerning quasi-Frobenius bimodules and depth two extensions. For instance, we give the endomorphism ring theorem for QF-bimodules and several characterizations of finite projective depht two extensions in terms of similar bimodules.

\begin{punto}{\bf On QF bimodules.}
For a unital bimodule ${}_BM_A$ over two unital rings $B$ and $A$, we let
$(_BM)^*$ denote its left $B$-dual and $(M_A)^*$ its right $A$-dual.
A unital ring homomorphism $\varphi: B \rightarrow A$ is referred to
as a \textit{ring extension} (of $A$ over $B$, or $A / B$).
We say $(B,A)$-bimodules $M$ \textit{divides} $N$, or $M \| N$,
if and only if $M \oplus P \cong N^{(n)}$ for some complementary $(B,A)$-bimodule $P$
and direct sum power of $N$.  Of course, $N \sim M$ if $M \| N$ and
$N \| M$.  Since the notion of division of modules may be formulated in terms of split exact sequences of the form $0\rightarrow M\rightarrow N^{(n)}\rightarrow X\rightarrow 0$, similarity is clearly an equivalence relation,
which is preserved by functors with the property of preserving finite direct sums.

\noindent
Recall from \cite{BM} that $\varphi$ is a left QF-extension if $_BA$ is finitely generated projective and $A\|(_BA)^*$ as $(A,B)$-bimodules. Similarly, $\varphi$ is a right QF-extension if $A_B$ is finitely generated projective and $A\|(A_B)^*$ as $(B,A)$-bimodules. We easily conclude that $\varphi$ is a QF-extension (left and right extension) if and only if $_BA$ is finitely generated projective and  $A\sim (_BA)^*$ as $(A,B)$-bimodules.
\bigskip

\noindent
 Recall also from \cite{CCN} that a $(B,A)$-bimodule $M$ is called  \emph{quasi-Frobenius bimodule}, or QF bimodule, if both $_BM$ and $M_A$ are finitely generated projective and $(_BM)^*\sim \, (M_A)^*$ as
$(A,B)$-bimodules. It is easy to see
that $\varphi$ is a QF extension if and only if the
natural bimodule $_AA_B$ is a QF-bimodule. In more detail,
$\Leftarrow$ is seen
from $A \sim (_BA)^*$ as $(A,B)$-bimodules,
so $_BA$ is finitely generated projective and $A \| (_BA)^*$, so  $\varphi$ is by definition right QF. It follows that ${}A_B$ is
finitely generated projective. To the last similarity, we apply the functor $(_B-)^*$ to obtain
$(_BA)^* \sim A$ as $(A,B)$-bimodules, so $A \| (A_B)^*$ and by definition  $\varphi$ is  left QF.

\begin{remark*}
In \cite{CCN} the notion of QF Frobenius extension was
extended to functors. In particular, $\varphi$ is a
QF extension if and only if the restriction of scalars functor
$\varphi_{*}$ is a QF functor. More generally, the bimodule $_BM_A$
is a QF bimodule if and only if the functor $M\otimes_A-$ is a QF
functor \cite[Proposition 3.5]{CCN}.
\end{remark*}
The calculus of similar bimodules lends itself to easy proofs of several results, e.g. of  an endomorphism ring theorem for the QF property. For instance, if ${}_BM_A$ is a bimodule and  $\lambda: B
\rightarrow E = \END M_A$ the left regular representation given by
$\lambda_b(m) = bm$, then the endomorphism ring theorem for QF bimodules is stated as follows.
\begin{proposition*}
If ${}_BM_A$ is a QF bimodule, then $B
\stackrel{\lambda}{\longrightarrow} E$ is a QF extension.
\end{proposition*}
\begin{proof}
Since ${}_BM$ is finite projective, so is ${(_BM)^*}_B$, and
${(M_A)^*}_B$ from the hypothesis that $(M_A)^* \| (_BM)^*$ as right
$B$-modules. Since $M_A$ is finite projective and $E \cong M
\otimes_A (M_A)^*$ (also as $(E, B)$-bimodules), it follows that
$E_B$ is finite projective.

It remains to show that $(E_B)^* \sim E$ as natural $(B,
E)$-bimodules. We compute with the natural $(B,E)$-bimodule
structures:
\begin{eqnarray*}
(E_B)^* & \cong & \Hom_{-B} (M \otimes_A {(M_A)^*},\, B) \\
        & \cong & \Hom_{-A} (M,\, \Hom_{-B}((M_A)^*, B)) \\
         & \sim & \Hom_{-A} (M,\, \Hom_{-B}((_BM)^*, B)) \\
         & \cong & E
\end{eqnarray*}
since $({(_BM)^*}_B)^* \cong M$ follows for the reflexive module
${}_BM$.
\end{proof}
\end{punto}
\begin{punto}{\bf On depth two extensions.}
Recall that a ring extension $B \rightarrow A$ is \textit{right depth two}, or right D2, if the natural $(A,B)$-bimodules $A \otimes_B A$ and $A$ itself are
similar.  A \textit{left depth two} extension is defined oppositely:  $A / B$ is left D2
if $A^{\rm op} / B^{\rm op}$ is right D2. Also dual theorems for left D2 extensions may be deduced in this way.

As an example, an H-separable extension $A / B$ is (left and right) D2,
since  its defining condition is that $A$ and $A \otimes_B A$
be similar as natural $(A,A)$-bimodules \cite{hirata}. Other examples
are Hopf-Galois extensions, pseudo-Galois extensions, and faithfully flat
projective algebras (The definition of depth two
is sometimes extended
in an straightforward way to include examples of infinite index
subalgebras such as Hopf-Galois extensions with infinite dimensional
Hopf algebra). A Hopf subalgebra of a semisimple Hopf $\C$-algebra
is depth two precisely when it is normal.  
Just when a one-sided depth two ring extension is automatically two-sided is an interesting question of chirality \cite{kadison-coD2} known to be the case also for Frobenius extensions: below we show that more generally QF extensions satisfy this property.

We give first a proposition with several characterizations of
finite projective right D2 extensions in terms of similar bimodules.

\begin{proposition}
\label{prop-a}
  Suppose $A / B$ is a ring extension such that the natural modules
${}_BA$ and $A_B$ are finite projective.  Then the following
are equivalent and characterize right D2 extension:
\begin{enumerate}[(a)]
\item $A \sim A \otimes_B A$ as $(A,B)$-bimodules;
\item $A \sim \END {}_BA$ as $(B,A)$-bimodules;
\item $(A_B)^* \sim (A_B)^* \otimes_B (A_B)^*$ as $(B,A)$-bimodules.
\end{enumerate}
\end{proposition}
\begin{proof}
Conditions (a) $\Leftrightarrow$ (b) follow from \cite[Proposition\ 3.8]{kadison-cen}, which makes only use of ${}_BA$ finite projective
in proving (b) $\Rightarrow$ (a)  (Note that
any ring extension will satisfy $A \| \END {}_BA$ as $(B,A)$-bimodules
since right multiplication $\rho: A \rightarrow \END {}_BA$ is a split
monic of $(B,A)$-bimodules, split by $f \mapsto f(1)$. The natural $(A,A)$-bimodule structure on $\END {}_BA$ is
given by $x \cdot f \cdot y = \rho_y \circ f \circ \rho_x$
for $x,y \in A$).

\noindent
Condition (a) $\Rightarrow$ (c) using $A_B$ is finite projective:
we note the isomorphism of natural $(B,A)$-bimodules,
\begin{equation}
\label{eq: lemma}
(A_B)^* \otimes_B (A_B)^* \cong (A \otimes_B A_B)^*
\end{equation}
via $\alpha \otimes_B \beta \longmapsto ( x \otimes_B y \mapsto \alpha(\beta(x)y))$,
which follows from  applications of \cite[20.11]{A-F} and its dual.
In this case, condition (c) follows from applying to (a) the functor
$(-_B)^*$ from the category of $(A,B)$-bimodules into the category
of $(B,A)$-bimodules.

\noindent
The reverse implication, conditions (c) $\Rightarrow$ (a)
follows from the isomorphism~(\ref{eq: lemma}), together with
noting that $A_B$ and therefore $A \otimes_B A_B$ are finite projective,
and applying the functor $(_B-)^*$ to condition (c).
\end{proof}
\begin{remark}
The significance of the proposition is that the notion
of depth two may be extended to functors as follows.  Suppose $F: \mathcal{C}
\rightarrow \mathcal{D}$ is a functor between abelian categories with
left or right adjoint $G: \mathcal{D} \rightarrow \mathcal{C}$.
Say that $F$ is \textit{depth two} if $FGF \sim F$, where similarity of functors is defined in a straightforward generalization of the two notions of similarity used above. If $F$ is one of the three functors of restriction, induction or coinduction between module categories
over rings $B$ or $A$, then by the proposition any choice of $F$ recovers the notion of depth two for (finite projective) ring extensions.
For example, coinduction  is tensoring by the dual in case
the ring extension is finite projective.
\end{remark}

We next show that QF extensions are two-sided depth two if one-sided.
We first establish a lemma about QF extensions, which is well-known
for Frobenius extensions (where isomorphism replaces similarity in
the conclusion).

\begin{lemma}
\label{lem-a}
Let  $A / B$ be a QF ring extension.  Then
the natural $(A,A)$-bimodules $\END {}_BA$,
$A \otimes_B A$ and $\END A_B$ are similar.
Moreover, $A \otimes_B A \sim E$ as natural $(E,A)$-bimodules
where $E = \END A_B$.
\end{lemma}
\begin{proof}
Assume that $A\|B$ is a QF ring extension. Then $A \sim (A_B)^*$ as $(B,A)$-bimodules.  Apply to this the functor $A \otimes_B -$ into the category of $(E,A)$-bimodules, obtaining $A \otimes_B A \sim \END A_B$, since it follows from $A_B$ finite projective that $\END A_B \cong A \otimes_B (A_B)^*$.  The rest of the proof is in the same vein.
\end{proof}
\begin{proposition}\label{converse}
Let $A / B$ be a QF ring extension.  Then $A / B$ is left D2
$\Leftrightarrow$ extension $A / B$ is right D2.
\end{proposition}
\begin{proof}
Suppose QF extension $A / B$ is right D2, so that $A \sim \END {}_BA$ as $(B,A)$-bimodules by Proposition~\ref{prop-a}. By Lemma~\ref{lem-a}, $A \sim A \otimes_B A$
as $(B,A)$-bimodules, whence $A / B$ is left D2.  The converse follows
from dualizing this argument.
\end{proof}

With a bit more care, it may be shown that if a left QF extension
is right D2, then it is left D2. Next we show a converse to the endomorphism
ring theorem for the property D2. Let $E$  denote $\END A_B$ of
a ring extension $A / B$.

\begin{theorem}
Let $A / B$ be a  QF ring extension and
$A$ is a generator as a right $B$-module.  If $A \stackrel{\lambda}{\longrightarrow} E$ is D2, then $A / B$ is D2.
\end{theorem}
\begin{proof}
Since $A_B$ is a progenerator, the rings $E$ and $B$ are Morita equivalent,
with context bimodules ${}_EA_B$ and ${}_B{(A_B)^*}_E$.
In particular, $(A_B)^* \otimes_E A \cong B$ as $(B,B)$-bimodules.
\noindent
Given the right D2 condition ${}_EE \otimes_A E_A \sim {}_EE_A$, we make
the substitution $E \sim A \otimes_B A$ as $(E,A)$-bimodules from
the lemma.  Thus, $A \otimes_B A \otimes_B A \sim A \otimes_B A$
as $(E, A)$-bimodules.  Now apply the functor $(A_B)^* \otimes_E -$
to obtain $A \otimes_B A \sim A$ as $(B,A)$-bimodules;
whence $A / B$ is left D2.  From Proposition \ref{converse}, $A / B$ is also
right D2.
\end{proof}
\end{punto}
\begin{punto}{\bf Depth three and more.} From the point of view of
finite depth and Galois correspondence, it is necessary to
generalize depth two to a tower of three rings $C \subseteq B
\subseteq A$, or more generally $C \rightarrow B \rightarrow A$
(denoting unital ring homomorphisms).  Recall the tower $A / B / C$
is \textit{right depth three}, or right D3, if $A \otimes_B A \sim A$ as
natural $(A,C)$-bimodules. Thus, with $B = C$, we recover the notion
of right D2 ring extension $A / B$.
\noindent
As an example of right D3 tower, let $A /
B / C$ be the group algebras of a tower of groups $G > H > K$ over
any commutative ground ring. Suppose the normal closure of $K$ in
$G$ is contained in $H$: $ K^G < H.$ Then $A / B / C$ is (left and
right) D3.  (Similarly to Proposition~\ref{converse}, we may show that a
QF tower of rings is two-sided D3 if one-sided.)

Depth three and more is originally an analytic notion for subfactors
using the basic construction.  Recall
 then from \cite{kadison-fd} that a ring extension
$A / B$ is \textit{right depth three}, or right D3 extension, if the tower of rings
$E / A/ B$ is right D3, where $E$ denotes $\END A_B$ and the default mapping
is as usual $A \stackrel{\lambda}{\longrightarrow} E$. For example, when $B$ and $A$ are the group algebras of a finite subgroup pair $H < G$, the
ring extension $A / B$ is right D3 if $H$ has a
normal subgroup complement in $G$.  
 As in subfactor
theory, there is an embedding theorem for  ring extensions that are
depth three into  ring extensions that are depth two; however,
we provide a purely algebraic proof using only the QF property of
ring extensions. Recall that a ring extension $A / B$ is a \textit{separable extension} if $A \otimes_B A$ contains a
(separability) element $e = e^1 \otimes_B e^2$ (possibly suppressing a summation over simple tensors) such that $e^1 e^2 = 1_A$ and
$ae = ea$ for all $a \in A$.

\begin{theorem*}
\label{th-conv}
Suppose $A / B$ is a separable QF-extension
and $E = \END A_B$.  Then $ A / B$ is
D3 $\Leftrightarrow$ the composite extension $E / B$ is D2.
\end{theorem*}
\begin{proof}
($\Rightarrow$) This part of the proof does not require separability.
We apply the bimodule similarity for the QF
extension $A / B$,  between its endomorphism ring
and  its tensor-square,
${}_AE_A \sim  {}_A A \otimes_B A_A$.
Tensoring by ${}_E E \otimes_A - \otimes_A E_A$,
we obtain $E \otimes_A E \otimes_A E \sim E \otimes_B E  $
as natural $(E,A)$-bimodules.  Now
 tensor by ${}_EE \otimes_A -$ the right D3 condition $E \sim E \otimes_A E$
to obtain ${}_E E \otimes_A E_B \sim {}_EE \otimes_B E_B$.  Again applying $E \sim E \otimes_A E$, we obtain
 ${}_EE \otimes_B E_B \sim
{}_E E_B$.
Thus $E/ B$ is  right D2.  Since it is
a QF extension as well by Proposition \ref{converse},
it is also left D2.

($\Leftarrow$) This part of the proof does not require $A / B$ be
a QF extension. Since $A / B$ is a separable extension, the
natural $(E,E)$-epimorphism $E \otimes_B E \rightarrow E \otimes_A E$
splits via a mapping  $x \otimes_A y \mapsto xe^1 \otimes_B e^2y$ where $e =
e^1 \otimes_B e^2$ denotes a separability element.  Thus $E \otimes_A E$
divides $E \otimes_B E$ which is similar to $E$ as $(E, B)$-bimodules.
It follows that $E \otimes_A E$ divides $E$ as $(E,B)$-bimodules,
which is a sufficient condition for $A / B$ to be right D3 extension (since
$E$  divides $E \otimes_A E$ via multiplication).
\end{proof}

Recall that higher depth is defined by iterating the endomorphism ring
construction.  Thus a QF extension $A / B$ is \textit{depth n} if
$E_{n-2} / E_{n-3} / B$ is a D3 tower, where $E_1 = E$, $E_{0} = A$,
and $E_{-1} = B$, and $E_m$ is the right endomorphism ring of the extension
$E_{m-2} \stackrel{\lambda}{\longrightarrow} E_{m-1}$ defined
inductively from $m \geq 2$. It is not hard to establish by similar means
to those above that a depth $n$ extension is also depth $n+1$.  From
this and a tunneling lemma we may establish
(by small modifications to the arguments
in \cite[section~8]{kadison-fd}) an embedding theorem
for a depth $n$ QF extension $A / B$; that $E_m / B$ is D2
for a sufficiently large $m \geq n-2$.
\end{punto}

\end{document}